\documentclass[12pt]{amsart}

\usepackage{times}      
\usepackage{amssymb}    
\usepackage{amsmath}    
\usepackage{amsthm}
\usepackage{tikz-cd}
\usepackage{placeins}
\usepackage{enumerate}

\usepackage{url}
\makeatletter
\def\url@leostyle{%
  \@ifundefined{selectfont}{\def\UrlFont{\sf}}{\def\UrlFont{\scriptsize\ttfamily}}}
\makeatother
\newtheorem{theorem}{Theorem}[section]
\newtheorem{corollary}[theorem]{Corollary} 
 
\newtheorem{proposition}[theorem]{Proposition} 
\newtheorem*{mainthm}{Main Theorem}

\theoremstyle{definition}
\newtheorem{definition}[theorem]{Definition}

\def\nnn{\mathbb{N}}
\def\rrr{\mathbb{R}}

\def\o{\mathcal{O}^{\cdot, D, \cdot}_{k, \cdot, v} (n)}
\def\a{\mathrm{Alex}^{\cdot, D, \cdot}_{k, \cdot, v} (n)}
\def\m{\mathcal{M}^{\cdot, D, \cdot}_{k, \cdot, v} (n)}
\def\on{\mathrm{O}(n)}
\def\meq{\mathcal{M}^{\mathrm{c}}_{\mathrm{eq}}}

\def\ali{^{\alpha}_i}
\def\pal{p^{\alpha}}
\def\pali{p^{\alpha}_i}
\def\ual{U^{\alpha}}
\def\uali{U^{\alpha}_i}

\def\cal{\tilde{U}^{\alpha}}
\def\cali{\tilde{U}^{\alpha}_i}

\DeclareMathOperator{\diam}{diam}
\DeclareMathOperator{\vol}{vol}

\def\dist{\textrm{dist}}
\def\embed{\hookrightarrow}

\def\co{\colon\thinspace}

\begin{document}

\title[Orbifold finiteness]{Orbifold finiteness under geometric and spectral constraints}

\author[Harvey]{John Harvey}
\address{University of Notre Dame, Department of Mathematics, Notre Dame, Ind. 46556, U.S.A.}
\email{jharvey2@nd.edu}

\subjclass[2010]{Primary: 53C23; Secondary: 53C20,57R18, 58J53} 

\date{\today}

\begin{abstract}
The class of Riemannian orbifolds of dimension $n$ defined by a lower bound on the sectional curvature and the volume and an upper bound on the diameter has only finitely many members up to orbifold homeomorphism. Furthermore, any class of isospectral Riemannian orbifolds with a lower bound on the sectional curvature is finite up to orbifold homeomorphism.
\end{abstract}

\maketitle

\section{Introduction}

The question of how the geometry of a Riemannian manifold controls its topology is of long-standing interest. One particular problem is finding geometric constraints to define a  class of manifolds which is finite up to homotopy, homeomorphism or diffeomorphism.

A convenient notation for these classes is to write $\mathcal{M}^{K, D, V}_{k, d, v} (n)$ for the class of all Riemannian manifolds $(M,g)$ with $k \leq \sec_g \leq K$, $d \leq \diam (M) \leq D$ and $v \leq \vol (M) \leq V$. Where a value is replaced with ``$\cdot$'' the condition is understood to be deleted.

The first such result is that of Weinstein, who showed that, for $\delta > 0$, $\mathcal{M}^{1, \cdot, \cdot}_{\delta, \cdot, \cdot} (2n)$, the class of unformly pinched positively curved manifolds of even dimension, has only finitely many members up to homotopy \cite{weinstein}.
Shortly after this, Cheeger showed that $\mathcal{M}^{K, D, \cdot}_{k, \cdot, v} (n)$ has finitely many members up to diffeomorphism  \cite{cheeger}. 

Grove and Petersen removed the upper bound on sectional curvature, and obtained finiteness of $\m$ up to homotopy \cite{gphomotopy}.
Shortly afterwards, in collaboration with Wu, this result was improved to show finiteness up to homeomorphism \cite{gpw}. 
As long as the dimension is not four, the work of Kirby and Siebenmann \cite{ks} implies finiteness up to diffeomorphism.

The homeomorphism finiteness result was generalized to Alexandrov geometry by Perelman with his Stability Theorem \cite{perstab}, which showed that $\a$, the corresponding class of Alexandrov spaces, is finite up to homeomorphism.

The present work generalizes the homeomorphism finiteness result of Grove, Petersen and Wu to the area of Riemannian orbifolds. An orbifold is a mild generalization of a manifold, and, to give just a few examples, the concept has found applications in Thurston's work on the Geometrization Conjecture \cite{thurston}, the construction of a new positively curved manifold by Dearricott \cite{Dear} and Grove--Verdiani--Ziller \cite{gvz}, and string theory, such as Dixon, Harvey, Vafa and Witten's conformal field theory built on a quotient of a torus \cite{dhvw}. The same convenient notation can be used for orbifolds, here replacing $\mathcal{M}$ with $\mathcal{O}$.

The first finiteness result for orbifolds is that of Fukaya \cite{fukaya}, who generalized the result of Cheeger, showing that a subclass of $\mathcal{O}^{K, D, \cdot}_{k, \cdot, v} (n)$ is finite up to orbifold diffeomorphism. 
Fukaya used a much more restrictive definition of orbifold, considering only the orbit spaces of global actions by finite groups on Riemannian manifolds. 
This corresponds to what Thurston called a ``good'' orbifold \cite{thurston}.

Working in dimension two, Proctor and Stanhope showed that $\mathcal{O}^{\cdot, D, \cdot}_{k, \cdot, v} (2)$ is finite up to orbifold diffeomorphism \cite{procstan}, providing a first generalization of the result of Grove, Petersen and Wu.
The homeomorphism finiteness result was then shown in all dimensions by Proctor, provided the orbifold has only isolated singularities \cite{proc}. 
Here the assumption that the only singularities are isolated is removed.

	\begin{mainthm}
	For any $k, D, v, n$, the class $\o$ has only finitely many members up to orbifold homeomorphism.
	\end{mainthm}

This result completes the generalization of Grove--Petersen--Wu's homeomorphism finiteness.

By Weyl's asymptotic formula, which Farsi has shown is valid for orbifolds \cite{farsi}, a Laplace isospectral class of orbifolds has fixed volume and dimension. Stanhope has shown that, in the presence of a lower bound on Ricci curvature, such a class has a uniform upper bound on its diameter \cite{stan}, and so, just as in \cite{proc}, the following corollary is clear.

\begin{corollary}
Any class of Laplace isospectral orbifolds with a uniform lower bound on its sectional curvature has only finitely many members up to orbifold homeomorphism.
\end{corollary}

This generalizes the similar result of Brooks, Perry and Petersen for Laplace isospectral manifolds \cite{bpp}. While one cannot hear the shape of an orbifold, one can, at least in the presence of a lower sectional curvature bound, know that there are only finitely many possibilities.

\section{Gromov--Hausdorff topologies}

A general approach for proving finiteness results such as the Main Theorem \cite{gpw,perstab,proc} is to proceed via a compactness or precompactness result for the class. 
A particularly useful topology (in fact, a metric) on the set of isometry classes compact metric spaces was proposed by Gromov \cite{gromov}. 
Gromov's metric generalizes the Hausdorff metric $H$ on the closed subsets of a compact metric space.

\begin{definition}Let $(X,d_X)$ and $(Y,d_Y)$ be metric spaces. A function $f \co X \to Y$ (not necessarily continuous) is called an Gromov--Hausdorff $\epsilon$--approximation if, for all $p,q \in X$, $\left| d_X(p,q) - d_Y(f(p),f(q)) \right| \leq \epsilon$ and an $\epsilon$--neighborhood of the image of $f$ covers all of $Y$.\end{definition}

\begin{definition}The \emph{Gromov--Hausdorff distance} between two compact metric spaces $(X, d_X)$ and $(Y, d_Y)$ is the infimum of the set of all $\epsilon$ such that there are Gromov--Hausdorff $\epsilon$--approximations $X \to Y$ and $Y \to X$.\end{definition}

The equivariant Gromov--Hausdorff topology was first defined by Fukaya \cite{fukaya}, and achieved its final form some years later in his work with Yamaguchi \cite{fyannals}. Consider the set of ordered pairs $(M, \Gamma)$ where $M$ is a compact metric space and $\Gamma$ is a closed group of isometries of $M$. Say that two pairs are equivalent if they are equivariantly isometric up to an automorphism of the group.
Let $\mathcal{M}^c_{eq}$ be the set of equivalence classes of such pairs.

\begin{definition}Let $(X,\Gamma),(Y, \Lambda) \in \meq$. An \emph{equivariant Gromov--Hausdorff $\epsilon$--approximation} is a triple $(f,\phi,\psi)$ of functions $f \co X \to Y$, $\phi \co \Gamma \to \Lambda$ and $\psi \co \Lambda \to \Gamma$ such that
	\begin{enumerate}
		\item $f$ is an Gromov--Hausdorff $\epsilon$--approximation;
		\item if $\gamma \in \Gamma, x \in X$, then $\dist(f(\gamma x) , \phi(\gamma) f(x) ) < \epsilon$; and
		\item if $\lambda \in \Lambda, x \in X$, then $\dist(f(\psi(\lambda) x) , \lambda f(x) ) < \epsilon$.
	\end{enumerate}
\end{definition}

Note that these functions need not be morphisms from the relevant category. The equivariant Gromov--Hausdorff distance is defined from these approximations just as with the standard Gromov--Hausdorff distance.

Convergence of non-compact spaces can also be defined by adding a basepoint. Such sequences are said to converge if the closed metric balls around the basepoint converge. Where equivariant  convergence of non-compact spaces is considered in the present work, the basepoint will always be fixed by the group. In this case, convergence also reduces to the convergence of closed balls.

By \cite[Proposition 3.6]{fyannals}, given a sequence in $\meq$, if the sequence of underlying metric spaces converges in the Gromov--Hausdorff topology to a compact metric space then there is a subsequence which converges in the equivariant Gromov--Hausdorff topology.

\section{Alexandrov geometry}

Certain curvature conditions define precompact subsets of the set of all compact metric spaces. 
For example, Gromov showed that the class of all Riemannian manifolds of dimension $n$, with diameter less than $D$, and with Ricci curvature greater than $(n-1)k$ is precompact \cite{gromov}. 
Strengthening the curvature condition to require a lower bound on the sectional curvature provides much more structure on the limit spaces, and it is in this context that Alexandrov geometry was first studied.

It is possible to show that, for a Riemannian manifold, the condition that sectional curvature be $\geq k$ can be expressed as a triangle-comparison condition.
Grove and Petersen showed \cite{gpjdg} that the closure of $\m$ is contained within the class of all complete length metric spaces satisfying this triangle-comparison condition.
It is natural, then, to study this class in its own right.

\begin{definition}
An \textit{Alexandrov space} of finite dimension $n\geq 1$ is a locally complete, locally compact, connected  length space, with a lower curvature bound in the triangle-comparison sense. By convention, a $0$--dimensional Alexandrov space is either a one-point or a two-point space.
\end{definition}

Many fundamental results in this area were proved by Burago, Gromov and Perelman \cite{bgp}. 
They showed that the class of all Alexandrov spaces is closed under passing to Gromov--Hausdorff limits, and under quotients by isometric group actions. 

Let $X$ be an Alexandrov space, and let $p \in X$. 
Then, also by \cite{bgp}, there is a uniquely defined tangent cone at $p$, $T_p X$, which can be obtained as a limit object by rescaling $X$ around $p$.
$T_p X$ is itself an Alexandrov space, with curvature $\geq 0$.

The most important singularities of an Alexandrov space are its extremal subsets, introducted by Perelman and Petrunin \cite{perpet}.
The distance functions in an Alexandrov space have well-defined gradients, and it is possible to flow along these gradients. 
The gradient flow gives a natural way to understand an extremal subset.

\begin{definition}Let $X$ be an Alexandrov space. A subset $E \subset X$ is extremal if, for every $p \in X$, the flow along the gradient of $\dist(p,\cdot)$ preserves $E$.\end{definition}

Trivial examples of extremal sets are the empty set, and the entire space $X$. Any point having a space of directions with diameter $\leq \pi/2$ is extremal, as is the boundary of an Alexandrov space. Where a compact Lie group acts on an Alexandrov space by isometries, the closure of the orbit-type strata in the orbit space are also extremal sets \cite{perpet}, an example of particular interest for the topic under discussion.

Extremal sets survive the passage to Gromov--Hausdorff limits, and so for any extremal set $E$, and any point $p \in E$, there is a well defined tangent subcone $T_p E \subset T_p X$ which is also extremal. Conversely, if $E$ is a closed subset of $X$ such that $T_p E$ is extremal for each $p \in E$, then $E$ is an extremal subset.

A crucial advance in the understanding of Alexandrov spaces was made by Perelman with his proof of the stability theorem \cite{perstab}. 
The author recommends the treatment by Kapovitch \cite{kapstab} for those who wish to learn more about this deep result. 
The statement of the theorem given here is a relative version of Perelman's original theorem. 
It was proved by Kapovitch for the case where only one extremal subset is under consideration, but as shown by Searle and the author \cite{HS}, it is in fact true in greater generality.

\begin{theorem}[Stability Theorem \cite{perstab, kapstab,HS}]\label{t:relstability}
	Let $X_i$ be a sequence of Alexandrov spaces of dimension $n$ with curvature uniformly bounded from below, converging to an Alexandrov space $X$ of the same dimension. 
	Let $\mathcal{E}_i = \{ E^{\alpha}_i \subset X_i \}_{\alpha \in A}$ be a family of extremal sets in $X_i$ indexed by a set $A$,  converging to a family of extremal sets $\mathcal{E}$ in $X$.
	
	Let $o(i) \co \nnn \to (0,\infty)$ be a function with $\lim_{i \to \infty} o(i) = 0$. Let $\theta_i\co  X \to X_i$ be a sequence of Gromov--Hausdorff $o(i)$--approximations. 
	
	Then for all large $i$ there exist homeomorphisms $\theta'_i \co  (X, \mathcal{E}) \to (X_i, \mathcal{E}_i)$, $o(i)$--close to $\theta_i$.
\end{theorem}

This result implies all the previously known finiteness results for manifolds, other than Cheeger diffeofiniteness in dimension four.
It also has a vital application in Alexandrov geometry.
Consider the construction of the tangent cone to an Alexandrov space by the convergence of the sequence obtained by rescaling the metric around a certain point.
By Theorem \ref{t:relstability}, the local structure of the space is controlled by the tangent cone.

\begin{corollary}Let $X$ be an Alexandrov space, and let $p \in X$. Then for some $r_0 >0$, $B_r(p) \cong T_p X$ for all $r < r_0$. Furthermore, $r_0$ and the homeomorphism can be chosen so that, for every extremal set $E$, $E \cap B_{r_0}(p)$ is mapped to $T_p E$.\end{corollary}

These small conical neighborhoods are extremely useful in the study of Alexandrov spaces, and so it will be convenient to make the following definition.

\begin{definition}
An open subset $U$ of an Alexandrov space $X$ is called \emph{cone-like around $p$} if $p \in U$, and there is a homeomorphism $f \co U \to T_p X$ with $f(p)$ being the vertex of the cone and $f(E \cap U) = T_p E$ for each extremal set $E$.
\end{definition}

Finally, the following result on equivariant convergence of Alexandrov spaces, which is due to the author, will be central in constructing orbifold category homeomorphisms. 
The original result is for equicontinuous sequences of actions.
This hypothesis is always satisfied for a finite group.

\begin{theorem}\cite{hequi}\label{t:equistab} 
	Let $G$ be a finite Lie group and let $(X_i,p_i)$ be a sequence of pointed Alexandrov spaces of dimension $n$ and curvature bounded below by $k$. Let $G$ act isometrically on each of $X_i$, fixing $p_i$. Suppose the sequence converges to an action of $\Gamma$ on another $n$-dimensional pointed Alexandrov space $(X,p)$ in the pointed equivariant Gromov--Hausdorff topology. 
	
	Then for large $i$ the spaces $X_{i}$ are equivariantly homeomorphic to $X$.
\end{theorem}

In the proof of that theorem, after passing to a subsequence, the convergent sequence is reformulated as a Hausdorff convergent sequence of invariant subspaces of an enveloping metric space $\mathcal{X}$ with an isometric action of $G$.
Then the orbit spaces converge in the Hausdorff sense inside $\mathcal{X}/G$.
If $f_i \co X/G \to X_i/G$ are Hausdorff approximations which are also homeomorphisms, then they can be lifted to equivariant homeomorphisms $F_i \co X \to X_i$.
This provides some flexibility in the construction of the $F_i$, since there is usually some freedom in the precise choice of $f_i$.

\section{Orbifolds}\label{ss:orbifolds}

Orbifolds were first introduced by Satake under the name \emph{V-manifolds} \cite{Satake}, as topological spaces locally modelled on a quotient of Euclidean space by a finite group. Some basic facts about orbifolds are reviewed here. The reader may refer to, among others, the book by Adem, Leida and Ruan \cite{alr} or Thurston's notes \cite{thurston} for further information.

\begin{definition}A smooth $n$--dimensional \emph{orbifold chart} over a topological space $U$ is a triple $(\tilde{U}, \Gamma_U, \pi_U)$ such that $\tilde{U}$ is a connected open subset of $\rrr^n$, $\Gamma_U$ is a finite group of smooth automorphisms of $\tilde{U}$ and $\pi_U \co \tilde{U} \to U$ is a $\Gamma_U$--invariant map inducing a homeomorphism $\tilde{U}/\Gamma_U \cong U$.\end{definition}

For convenience, a chart will sometimes be referred to as being over a point $p$. 
This will mean that the chart is over some neighborhood of $p$.

Let $U$ and $V$ be open subsets of a topological space $X$, and let $(\tilde{U}, \Gamma_U, \pi_U)$ and  $(\tilde{V}, \Gamma_V, \pi_V)$ be orbifold charts of dimension $n$ over $U$ and $V$ respectively.
The charts are called \emph{compatible} if, for every $p \in U \cap V$, there is a neighborhood $W$ of $p$ and an orbifold chart $(\tilde{W}, \Gamma_W, \pi_W)$ over $W$ such that there are smooth embeddings $\lambda_U \co \tilde{W} \embed \tilde{U}$ and $\lambda_V \co \tilde{W} \embed \tilde{V}$ with $\pi_V \circ \lambda_V = \pi_W$ and $\pi_U \circ \lambda_U = \pi_W$.

As usual, an orbifold atlas on a space $X$ will mean a collection of compatible charts covering $X$. Now the definition of an orbifold can be made.

\begin{definition}A smooth \emph{orbifold} of dimension $n$ is a paracompact Hausdorff space equipped with an atlas of orbifold charts of dimension $n$.\end{definition}

An orbifold diffeomorphism (respectively homeomorphism) is a homeomorphism of the underlying topological space which can locally be lifted to an equivariant diffeomorphism (respectively homeomorphism) of charts. 

Let $X$ be an orbifold, let $p \in X$, and let $(\tilde{U}, \Gamma, \pi)$ be a chart over $p$ with $\pi(y) = p$.
The isotropy group of $y$ will be called the \emph{local group} at $p$, and will be written as $\Gamma_p$. 
It is uniquely defined up to conjugacy in $\Gamma$, and choosing a different chart does not change the isomorphism type of the group. 

In fact, one can always choose a linear chart over $p$ such that the group of automorphisms is isomorphic to $\Gamma_p$. By this is meant a chart of the form $(\rrr^n, \Gamma_p, \pi)$ where the action of $\Gamma_p$ is via a faithful orthogonal representation $\rho_p \co \Gamma_p \embed \on$. Such a chart will be referred to as a \emph{linear chart around $p$}. The representation is also uniquely determined up to isomorphism, and will be called the \emph{local action} at $p$. The differential of the action of $\Gamma_p$ at the origin of the chart is also isomorphic to $\rho_p$.

A \emph{Riemannian metric} on an orbifold can be given by fixing a finite atlas and a partition of unity with respect to the corresponding cover, and choosing Riemannian metrics on the charts which are invariant with respect to the finite group action. An orbifold equipped with a Riemannian metric is called a Riemannian orbifold. Once the metric on the orbifold is given it can be lifted to the maximal atlas in a canonical manner. The various notions of curvature at points of an orbifold can then be defined by reference to the curvature of the charts.

It is straightforward to see that an orbifold with sectional curvature $\geq k$ is also an Alexandrov space with curvature $\geq k$.
The tangent cone at any point of an orbifold is then well-defined, and coincides with the usual notion of tangent space for orbifolds.
The notion of an extremal set now finds a very natural application in orbifolds.

\begin{proposition}\label{p:extremalstrata}Let $X$ be an orbifold of dimension $n$, $\Gamma$ a finite group, and $\rho \co \Gamma \embed \on$ a linear representation of $\Gamma$. Let $X^{\rho}$ be the closure of all points with local action $\rho$. Then $X^{\rho}$ is an extremal set of $X$.\end{proposition}

\begin{proof}The result is clear where $n = 1$. Let $p \in X^{\rho}$ and consider the local action at $p$ by $\Gamma_p$. The tangent cone at $p$ is the cone on the quotient of the unit sphere by $\Gamma_p$. Consider the image of those points in the unit sphere having isotropy isomorphic to $\rho$. The closure of the cone on this set is $T_p X^{\rho}$ and by induction it is extremal in $T_p X$.  Since $X^{\rho}$ is closed, it is extremal.\end{proof}

The following proposition now shows that a linear chart around $p$ can be extended over any cone-like set around $p$. 

\begin{proposition}\label{p:chart}Let $X$ be an orbifold, and let $p \in X$. Let $U$ be a cone-like set around $p$. Then there is a linear chart over $U$ around $p$.\end{proposition}
\begin{proof}
Consider the differential of the local action of $\Gamma_p$ on $\rrr_n$.
The quotient of this action is the tangent cone at $p$, $T_p X$.

Let $f \co U \to T_p X$ be a homeomorphism carrying each extremal set $E$ in $U$ to $T_p E$.
Note that because $U$ is cone-like, $f$ preserves the local action at every point.

Using a maximal atlas, cover $U$ by the ranges of all possible linear charts, $\{U_{\kappa}\}_{\kappa \in K}$. 
Discard any $U_{\kappa}$ such that $f(U_{\kappa})$ is not the range of a linear chart in $T_p X$.

Observe that this reduced family still covers $U$. 
Suppose some $q \in U$ is not in any element of the reduced family. 
Then for every $\kappa \in K$ such that $q \in U_{\kappa}$, $f(U_{\kappa})$ is not a linear chart.
But $f(q)$ is covered by \emph{some} linear chart, and the intersection $W$ of the range  of this chart with $f(U_{\kappa})$ is also covered by a linear chart.
Then because $f^{-1}(W) \subset U_{\kappa}$ it too is covered by a linear chart.
It follows that $f^{-1}(W)=U_{\lambda}$ for some $\lambda \in K$, and is in the reduced family.

Select a countable subcover, $U_1, U_2, \ldots$, and write $V_i$ for $f(U_i)$.
Let $\Gamma_i$ be the local group acting on the charts $U_i$ and $V_i$.
The charts $\tilde{V}_1, \tilde{V}_2, \ldots$ can be glued together to construct a chart over all of $T_p X$.
The gluing requires $[\Gamma_p : N(\Gamma_i)]$ copies of $\tilde{V}_i$.
The manner of this gluing gives a set of instructions which allows one to glue the charts $\tilde{U}_1, \tilde{U}_2, \ldots$ together to obtain the desired chart $\tilde{U}$.

Since this chart is built by gluing together charts from the orbifold atlas, it is compatible with the atlas.
\end{proof}

\section{Proof of Main Theorem}

By the Stability Theorem \ref{t:relstability}, $\o$ contains only finitely many topological types. To prove the main theorem, it is therefore sufficient to prove the following.

\begin{theorem}Let $X$ be a compact topological space. Then, up to orbifold homeomorphism, there are only finitely many orbifold structures on $X$ which belong to $\o$.\end{theorem}

\begin{proof}
Aiming for a contradiction, let $O_i$ be a sequence of orbifolds in $\o$, all of which have underlying topological space $X$, and no two of which are orbifold homeomorphic. 
By compactness of $\a$, a subsequence of $O_i$ converges in the Gromov--Haudorff sense to some $Y \in \a$ which also has underlying space $X$. 
Abusing notation, the subsequence will still be written as $O_i$.
Since there will be many more instances of passing to subsequences, this abuse of notation will be repeated throughout the proof.

By Stanhope \cite{stan} there is a uniform upper bound on the order of the local group of a point in $\o$. 
Recall a finite group has only finitely many linear representations in a given dimension.
It follows that all the possible local actions up to isomorphism can be listed by $\rho_j \co G_j \to \mathrm{GL}(n)$, for $j = 1 , \ldots , m$ where $m$ is some finite number.
Let $E_i^j$ be $O_i^{\rho_j}$, the closure of the subset of $O_i$ with local group action isomorphic to $\rho^j$.
By Proposition \ref{p:extremalstrata} the $E_i^j$ are extremal sets.

Passing to a subsequence $m$ times if necessary, one may assume that each sequence $E_i^j$ converges to an extremal subset $E^j \subset Y$. 
Now, by the relative stability theorem \cite{HS}, there are homeomorphisms $h_i \co Y \to O_i$ which are Gromov--Hausdorff approximations and carry each of the $E^j$ onto the $E_i^j$. 

To prove the result, it is now sufficient to show that $h_{ij} \co O_i \to O_j$ given by $h_{ij} = h_j \circ h_i^{-1}$ is an orbifold homeomorphism.

Let $\pal$ be a set of points in $Y$ such that $Y$ is covered by cone-like metric balls $\ual$ centered at $\pal$. 
Then the sets $h_i (\ual)$ are also cone-like around $\pali = h_i (\pal)$, and cover $O_i$. 
Denote these sets by $\uali$.

By Proposition \ref{p:chart} each $\uali$ is covered by a chart $(\cali,\Gamma_{\pali},\pi_{\uali})$. 
By passing to a subsequence, we may assume that the $\cali$ form a convergent sequence in the pointed equivariant Gromov--Hausdorff topology, converging to some object $(\cal,\Gamma_{\pal}) \in \mathcal{M}^c_{eq}$.

Now, by Theorem \ref{t:equistab}, $\cali$ and $\cal$ are equivariantly homeomorphic by some $F_i \co \cal \to \cali$. 
The $F_i$ induce homeomorphisms  $f_i \co \cali / \Gamma_{\pali} \to \cal / \Gamma_{\pal}$ which are Hausdorff approximations witnessing the Hausdorff convergence of the orbit spaces inside the enveloping orbit space. 

Write $\mu\ali$ for the isometry $\cali / \Gamma_{\pali} \to \uali$ induced by $\pi_{\uali}$.

\begin{center}
\begin{tikzcd}
	\cali 					\arrow[two heads]{d}	& \cal \arrow{l}{F_i}[swap]{\cong} \arrow[two heads]{d} \\
	\cali / \Gamma_{\pali}	\arrow{d}{\mu\ali}[anchor=center,rotate=-90,yshift=-1ex,xshift=-0.2ex]{\cong}	& \cal / \Gamma_{\pal} \arrow{l}{f_i}[swap]{\cong}\\
	\uali											& \ual \arrow{l}{h_i}[swap]{\cong}
\end{tikzcd}
\end{center}

Now the gap may be filled in by a homeomorphism $\phi_i \co \cal / \Gamma_{\pal} \to \ual$ given by $h_i ^{-1} \circ \mu\ali \circ f_i$. The $\phi_i$ make up a sequence of Gromov--Hausdorff approximations, and the sequence converges to some isometry $\phi \co \cal / \Gamma_{\pal} \to \ual$. Then the $f_i$ may be adjusted slightly, setting $f_i = (\mu\ali)^{-1} \circ h_i \circ \phi$, and the $F_i$ adjusted to induce the new $f_i$.

This gives a non-smooth orbifold chart over $\ual$, $(\cal, \Gamma_{\pal}, \phi)$ such that the $h_i \co \ual \to \uali$ are orbifold homeomorphisms. The maps $h_{ij}$ are then also orbifold homeomorphisms.
\end{proof}

\FloatBarrier

\bibliographystyle{habbrv}
\bibliography{orbifold}

\end{document}